\documentclass[letterpaper, 10 pt, conference]{ieeeconf}  

\pdfoutput=1
\IEEEoverridecommandlockouts                              
\usepackage{amssymb}
\usepackage{graphicx}
\graphicspath{{./figs/}}
\usepackage{float}
\usepackage{color}
\usepackage{mathrsfs} 
\usepackage{mathtools}
\usepackage{caption}
\usepackage{subcaption}
\usepackage{cite}
\usepackage{bm}
\usepackage{hyperref}
\usepackage{cleveref}
\usepackage[table]{xcolor}

\usepackage{algorithm, algorithmicx, algpseudocode}
\allowdisplaybreaks

\usepackage{amsthm}

\theoremstyle{plain}
\newtheorem{theorem}{Theorem}
\newtheorem{lemma}[theorem]{Lemma}

\theoremstyle{definition}
\newtheorem{remark}[theorem]{Remark}

\newcommand{\RR}{\mathbb{R}}
\newcommand{\NN}{\mathbb{N}}

\newcommand{\setN}{\mathcal{N}}
\newcommand{\setK}{\mathcal{K}}
\newcommand{\setT}{\mathcal{T}}
\newcommand{\setV}{\mathcal{V}}

\newcommand{\blkdiag}{\mathrm{blkdiag}}
\newcommand{\col}{\mathrm{col}}

\title{\LARGE \bf Structured linear quadratic control computations over 2D grids}

\author{Armaghan Zafar and Ian R. Manchester
	\thanks{The authors are with the Australian Centre for Field Robotics, and Sydney Institute for Robotics and Intelligent Systems, The University of Sydney, Sydney, NSW 2006, Australia
		(e-mail: {\tt\small  armaghan.zafar. ian.manchester@sydney.edu.au}).}%
}

\graphicspath{{./figures/}}

\begin{document}

\maketitle
\pagestyle{empty}

\begin{abstract}
    In this paper, we present a structured solver based on the preconditioned conjugate gradient method (PCGM) for solving the linear quadratic (LQ) optimal control problem for $K \times N$ sub-systems connected in a two-dimensional (2D) grid structure. Our main contribution is the development of a structured preconditioner based on a fixed number of inner-outer iterations of the nested block Jacobi method. We establish that the proposed preconditioner is positive-definite. Moreover, the proposed approach retains structure in both spatial dimensions as well as in the temporal dimension of the problem. The arithmetic complexity of each PCGM step scales as $O(KNT)$, where $T$ is the length of the time horizon. The computations involved at each step of the proposed PCGM are decomposable and amenable to distributed implementation on parallel processors connected in a 2D grid structure with localized data exchange. We also provide results of numerical experiments performed on two example systems.
\end{abstract}

\section*{Notations}\label{sec:notations}
	In this paper, we use $\mathbb{R}:=(-\infty,+\infty)$ and $\mathbb{N}:=\{1,2,\ldots\}$ to denote the set of real numbers and natural numbers, respectively. We denote an $n$-dimensional real vector with $\mathbb{R}^n$ and a real matrix with $n$ rows and $m$ columns with $\mathbb{R}^{n\times m}$. $\mathcal{A} \backslash\{b\}$ denotes all elements of the set $\mathcal{A}$ except element $b$. $A^\prime$ denotes the transpose of a matrix. $A \succ 0$ means the symmetric matrix $A=A^\prime \in\mathbb{R}^{n\times n}$ is positive definite (i.e., there exists $c>0$ such that $x^\prime A x \geq c x^\prime x$ for all $x\in\mathbb{R}^n$) and $A \succeq 0$ means $A$ is a positive semi-definite (i.e., $x^\prime A x \geq 0$ for all $x\in\mathbb{R}^n$). ``$\mathrm{blkdiag}(.)$"" represents construction of block diagonal matrix from input arguments and ``$\mathrm{col}(.)$" represents concatenation of input arguments as column vector.

\section{Introduction}\label{sec:introduction}
    We are interested in solving finite-horizon linear-quadratic (LQ) optimal control problems with discrete-time dynamics arising from the interconnection of sub-systems in a two-dimensional (2D) grid graph structure as shown in Fig.\ref{fig:2d_graph}.
	\begin{figure}
		\centering
		\includegraphics[width = 0.46\textwidth]{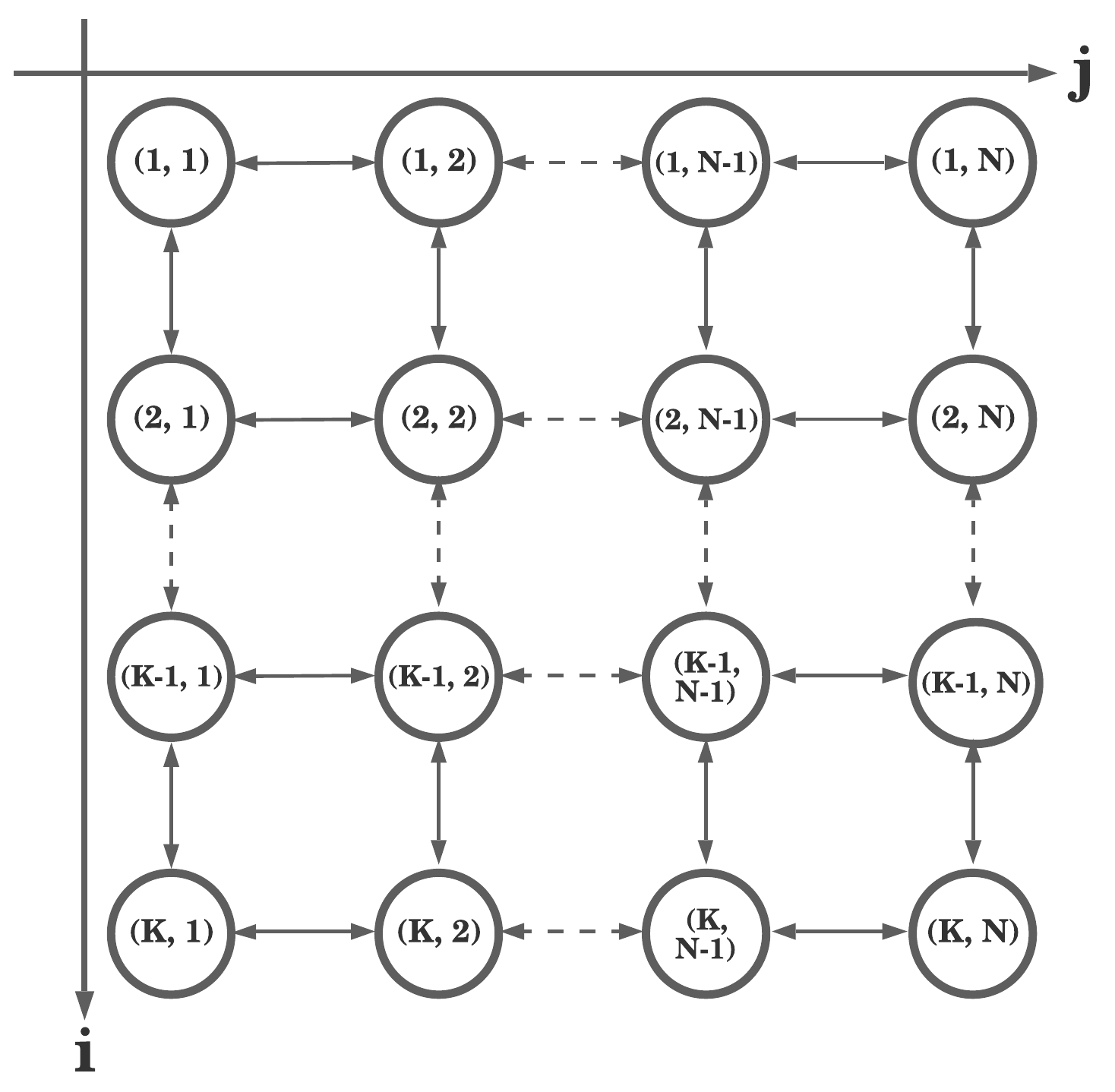}
		\caption{Bi-directional flow of information between subsystems connected in a 2D grid}
		\label{fig:2d_graph}
	\end{figure}
	Each node is a sub-system and is labeled with subscript $(i,j)$ with $i \in \setK := \{1,2,\ldots ,K\} \subset\NN $ and $j~\in~\setN:= \{1,2,\ldots, N\} \subset\NN$, where $K$ is the total number of sub-systems along the vertical direction and $N$ is the total number of sub-systems in the horizontal direction.
	We assume that there is a bi-directional flow of information between adjacent sub-systems, as shown in Fig.\ref{fig:2d_graph}. That is, each sub-system $(i,j)$ is coupled, at max, with $4$ of the neighbouring sub-systems. 
	Spatio-temporal structure of this kind is relevant in many applications such as automated water irrigation networks \cite{Li2005water}, radial power distribution networks \cite{Peng2018distributed}, control of traffic signals \cite{Liu2022optimal}, spreading processes such as wildfires \cite{Karafyllidis1997model, Somers2019priority} and in the discretization of 2D partial differential equations \cite{Maurer2000optimization}.

    The LQ optimal control problem that we are interested in is a large-scale quadratic program with O(KN T) decision variables, where T is the length of the time horizon. We can solve this problem by solving the linear system of equations that arise from the corresponding first-order optimality conditions, known as KKT conditions. Due to the special spatio-temporal structure of the problem, this linear system of equations is also highly structured.
	%
    %
	%

    Previous research has studied the temporal structure of the LQ optimal control problem in \cite{Rao1998application, Wang2010fast, Shahzad2012stable, Nielsen2019direct}, but these approaches do not utilize the structure in the spatial dimensions of the problem. Structured solvers for LQ optimal control problems related to path-graph networks (i.e., 1D chains) are proposed in \cite{Rey2020admm,  Zafar2019optimal, Zafar2022structured}, and the computations are decomposable and amenable to distributed implementation on parallel processors with path-graph data exchange. However, these approaches do not exploit the structure in both spatial dimensions simultaneously, as we do in this paper. In \cite{Nedic2010constrained, Falsone2017dual}, structured methods are proposed for solving optimal control problems of multi-agent networks, which can be used to solve the LQ optimal control problem we are interested in. However, these approaches often take a large number of iterations to converge. All of these approaches offer decomposable computations, but they do not exploit the complete three-dimensional spatio-temporal structure of the problem. In \cite{Zafar2020linear}, a structured solver is proposed for a special case where sub-systems are connected in a directed tree graph structure.
    
	The main contribution of this paper extends the developments presented in \cite{Zafar2022structured} to the case of 2D grids. We present a structured solver based on preconditioned conjugate gradient method (PCGM) for solving the KKT conditions. Specifically, we propose a structured preconditioner based on nested block Jacobi method (NBJM) to improve the conditioning of the linear system. The proposed approach retains structure in all three dimensions of the problem, extending the work in \cite{Zafar2022structured} which explored structure in two dimensions. We show that the proposed preconditioner is positive-definite. Moreover, the arithmetic complexity of each PCGM step scales as $O(KNT)$. The computations are decomposeable and amenable to distributed implementation on $O(NT)$ parallel processors connected in a 2D grid with localized data exchange.
	
	The rest of the paper is organized as follows. In Section~\ref{sec:problem_formulation}, we present the structured formulation of the problem and the corresponding KKT conditions. In Section~\ref{sec:pcgm}, we provide an overview of the PCGM. The structured preconditioner based on NBJM is developed in Section~\ref{sec:nested_jacobi_preconditioner}. We explore performance of the proposed PCGM numerically, for a 2D mass-spring-damper network and an automated water irrigation network, and present results in Section~\ref{sec:numerical_experiments}. Finally, some concluding remarks and possible directions of future work are presented in \ref{sec:conclusions}.

\section{Problem Formulation}\label{sec:problem_formulation}
    The state space dynamics of each sub-system $(i,j)$ are
    \begin{equation}\label{eq:2d_graph_dynamics}
		\begin{aligned}
		      x_{i,j,t+1} =~ &A_{i,j,t}x_{i,j,t} + B_{i,j,t}u_{i,j,t}\\
            &+ E_{i,j,t}x_{i,j-1,t} + F_{i,j,t}x_{i,j+1,t}\\
            &+ G_{i,j,t}x_{i-1,j,t} + H_{i,j,t}x_{i+1,j,t}
		\end{aligned}
    \end{equation}
    where $x_{i,j,t}\in\RR^{n_{i,j}}$ and $u_{i,j,t}\in\RR^{m_{i,j}}$  are the state and input of sub-system $(i,j)\in\setK \times \setN$ at time $t\in\setT:=\{0,1,\ldots,T\}\subset\NN\cup\{0\}$, respectively. Initial conditions are given by $x_{i,j,0} = \gamma_{i,j} \in \RR^{n_{i,j}}$ and the spatial boundary conditions are given by
    $x_{0,j,t}=\underline{\alpha}_{j,t} \in\RR^{n_{0,j}}$, $x_{K\!+\!1,j,t}=\overline{\alpha}_{j,t} \in\RR^{n_{K\!+\!1},j}$ and
    $x_{i,0,t}=\underline{\beta}_{i,t} \in\RR^{n_{i,0}}$, $x_{i,N\!+\!1,t}=\overline{\beta}_{i,t} \in\RR^{n_{i,N\!+\!1}}$ for $t\in\setT$.
 
	We are interested in the following finite-horizon linear-quadratic (LQ) optimal control problem:
	\begin{subequations}\label{eq:LQ_problem}
		\begin{equation}\label{eq:cost_function}
			\min_{x_{i,j,t},~u_{i,j,t}
			}\quad 
			\frac{1}{2}
			\sum_{t\in\setT}
			\sum_{j\in\setN}
			\sum_{i\in\setK}
			\ell_{i,j,t}(            
			x_{i,j,t},
			u_{i,j,t})
		\end{equation}
		\text{subject to}
		\begin{align}
			\eqref{eq:2d_graph_dynamics}
			~\text{ for }
			(i,j,t)\in\setK~\times&~\setN\times(\setT\backslash\{T\}),
			\label{eq:equality_constraint}\\
			x_{0,j,t}=\underline{\alpha}_{j,t}, 
			x_{K\!+\!1,j}=\overline{\alpha}_{j,t}
			&~\text{ for } j \in \setN,~t \in \setT,                  \label{eq:boundary_condition_i}\\   
			x_{i,0,t}=\underline{\beta}_{i,t},
			x_{i,N\!+\!1,t}=\overline{\beta}_{i,t}
			&~\text{ for } i\in\setK,~t\in\setT,
			\label{eq:boundary_condition_j}\\	
			x_{i,j,0} = \gamma_{i,j} 
			&~\text{ for } i \in \setK,~j \in\setN, 
			\label{eq:initial_condition}
		\end{align}
	\end{subequations}
	where $\ell_{i,j,t}(x,u) = x^\prime Q_{i,j,t} x + u^\prime R_{i,j,t} u$. For $i\in\setK$, $j\in\setN$ and
	$t\in\setT\backslash\{T\}$, it is assumed that
	$ Q_{i,j,t} = Q_{i,j,t}^{\prime} \succ 0$, and
	$R_{i,j,t} = R_{i,j,t}^{\prime}\succ 0$.
	Moreover, for every $i\in\setK$, $j\in\setN$, $Q_{i,j,T}\succ 0$, but $R_{i,j,T}=0$, so that $u_{i,j,T}$ can be removed as a decision variable. 
    
    Note that the problem \eqref{eq:LQ_problem} is a large-scale LQ optimal control problem with $O(KNT)$ decision variables. While the cost is separable across the sub-systems in both dimensions as well as the time horizon, the problem is not separable due the spatial coupling between the states of neighbouring sub-systems, and the inter-temporal coupling in the equality constraint \eqref{eq:equality_constraint}.
    %
    Solution of problem~\eqref{eq:LQ_problem} can obtained by solving the large-scale linear system of equations arising from the corresponding first-order optimality conditions known as KKT conditions.
    However, due to the structured coupling of variables (i.e., states) in the equality constraint \eqref{eq:equality_constraint}, this linear system of equations is also structured.

    \subsection{Structured Reformulation}\label{subsec:structured_reformulation}
        By defining $\bar{m}_j = \sum_{i=1}^{K}(m_{i,j})$, $\bar{n}_j = \sum_{i=1}^{K}(n_{i,j})$, $\bar{u}_{j,t} = \col(u_{1,j,t}, \ldots, u_{K,j,t}) \in
        \RR^{\bar{m}_j}$, and
        $\bar{x}_{j,t} = \col(x_{1,j,t}, \ldots, x_{K,j,T}) \in
        \RR^{\bar{n}_j}$,
        problem \eqref{eq:LQ_problem} can be
        reformulated as the following quadratic program:
        \begin{subequations}\label{eq:v_stacked_LQ_problem}
            \begin{equation}\label{eq:v_stacked_cost_function}
                \min_{\bar{x}_{j,t},~\bar{u}_{j,t}}
                \quad 
                \frac{1}{2}
                \sum_{t\in\setT}
                \sum_{j\in\setN}
                \bar{\ell}_{j,t}(            
                  \bar{x}_{j,t},
                  \bar{u}_{j,t})
            \end{equation}
            \text{subject to }
            \begin{align}
                \bar{x}_{j,t+1} &= \bar{A}_{j,t}\bar{x}_{j,t} + \bar{B}_{j,t}u_{j,t} + \bar{E}_{j,t}\bar{x}_{j-1,t} + \bar{F}_{j,t}\bar{x}_{j+1,t}\nonumber\\
                &+ \bar{M}_{j,t}{\boldsymbol{\alpha}}_{j,t},~j\in\setN,~t\in(\setT\backslash\{T\}), \label{eq:v_stacked_equality_constraint}\\
                \bar{x}_{0,t}&=\underline{\beta}_{t},~ \bar{x}_{N\!+\!1,t} =\overline{\beta}_{t},~ t\in\setT,
                \label{eq:v_stacked_boundary_condition}\\	
                \bar{x}_{j,0} &= \gamma_{j},~j \in\setN, 
                \label{eq:v_stacked_initial_condition}
            \end{align}
        \end{subequations}
        where $\bar{\ell}_{j,t}(\bar{x},\bar{u}) = \bar{x}^\prime \bar{Q}_{j,t} \bar{x} + \bar{u}^\prime \bar{R}_{j,t} \bar{u}$ for $j\in\setN$ and
        $t\in\setT$, $\bar{R}_{j,T} = 0$,
        \begin{align*}
            \bar{Q}_{j,t} &= \blkdiag(Q_{1,j,t}, \ldots, Q_{K,j,t}) 
                                    \in \RR^{\bar{n}_{j}\times\bar{n}_j},\\ 
            \bar{R}_{j,t} &= \blkdiag(R_{1,j,t}, \ldots, R_{K,j,t}) 
                                       \in \RR^{\bar{m}_{j}\times\bar{m}_j},\\ 
            \bar{B}_{j,t} &= \blkdiag(B_{1,j,t}, \ldots, B_{K,j,t})
                                       \in  \RR^{\bar{n}_{j}\times\bar{m}_j},\\ 
            \bar{E}_{j,t} &= \blkdiag(E_{1,j,t}, \ldots, E_{K,j,t}) 
                                       \in \RR^{\bar{n}_{j}\times\bar{n}_{j-1} },\\ 
            \bar{F}_{j,t} &= \blkdiag(F_{1,j,t}, \ldots, F_{K,j,t}) 
                                       \in \RR^{\bar{n}_{j}\times\bar{n}_{j+1} },\\ 
            \underline{\beta}_{t} &= \col(\underline{\beta}_{1,0,t},\ldots,
                                            \underline{\beta}_{K,0,t}) \in\RR^{\bar{n}_{0}},\\ 
            \overline{\beta}_{t} &= \col(\underline{\beta}_{1,N+1,t},\ldots,
                                            \underline{\beta}_{K,N+1,t}) \in\RR^{\bar{n}_{N+1}},\\ 
            {\boldsymbol{\alpha}}_{j,t} &= \col(\underline{\alpha}_{j,t},\overline{\alpha}_{j,t}) \in
                                 \RR^{n_{0,j}+n_{K+1,j}},\\ 
            \bar{A}_{j,t} &\!=\! 
            \renewcommand{\arraystretch}{1,2}
            \setlength{\arraycolsep}{0.1pt}
            \begin{bmatrix}
            A_{1,j,t}      &H_{1,j,t}       &          &\\
            G_{2,j,t} &A_{2,j,t}     &\ddots          &\\
            &\ddots &\ddots    &H_{K\!-\!1,j,t}\\
            & 		&G_{K,j,t} &A_{K,j,t}
            \end{bmatrix}, 
            \bar{M}_{j,t}\!=\! 
            \renewcommand{\arraystretch}{1.2}
            \setlength{\arraycolsep}{0.1pt}
            \left[
            \begin{array}{cc}
            G_{1,j,t} 	 &0\\
            0            &\vdots\\
            \vdots       &0\\
            0 		 	 &H_{K,j,t}
            \end{array}
            \right]
        \end{align*}%
        Note that $\bar{A}_{j,t} \in \RR^{\bar{n}_{j}\times\bar{n}_{j}}$ and $\bar{M}_{j,t} \in \RR^{\bar{n}_{j}\times(n_{0,j}+n_{K+1,j})}$. All the matrices in vertically stacked QP \eqref{eq:v_stacked_LQ_problem} are block diagonal except $\bar{A}_{j,t}$ which is block tri-diagonal due to the spatial coupling of system dynamics along the vertical direction in the optimal control problem \eqref{eq:LQ_problem}.
        
        Now, defining $\hat{m} = \sum_{i=1}^{N}(\bar{m}_{j})$, $\hat{n} = \sum_{i=1}^{N}(\bar{n}_{j})$, $\hat{u}_{t} = \col(u_{1,t}, \ldots, u_{N,t}) \in
        \RR^{\hat{m}}$, and $\hat{x}_{t} = \col(x_{1,t}, \ldots, x_{N,T}) \in
        \RR^{\hat{n}}$,
        problem \eqref{eq:v_stacked_LQ_problem} can be reformulated by stacking across the horizontal dimension of the 2D grid:
        \begin{subequations}\label{eq:h_stacked_LQ_problem}
            \begin{equation}\label{eq:h_stacked_cost_function}
                \min_{\hat{x}_{t},~\hat{u}_{t}}
                \quad 
                \frac{1}{2}\biggl( 
                \sum_{t\in\setT\backslash\{T\}}
                \begin{bmatrix}
    					\hat{x}_t\\
    					\hat{u}_t
    				\end{bmatrix}^{\prime}
    				\begin{bmatrix}
    					\hat{Q}_{t} & 0\\
    					0		   & \hat{R}_{t}
    				\end{bmatrix}
    				\begin{bmatrix}
    					\hat{x}_t\\
    					\hat{u}_t
    				\end{bmatrix}\biggr)
                + \hat{x}_{T}^\prime \hat{Q}_{T} \hat{x}_{T}
            \end{equation}
            \text{subject to }
            \begin{align}
                \hat{x}_{t+1} &= \hat{A}_{t}\hat{x}_{t} + \hat{B}_{t}\hat{u}_{t} + \hat{M}_{t}\hat{\boldsymbol{\alpha}}_{t}
                + \hat{N}_{t}\hat{\boldsymbol{\beta}}_{t},~t\in(\setT\backslash\{T\}), \label{eq:h_stacked_equality_constraint}\\
                \hat{x}_{0} &= \hat{\boldsymbol{\gamma}}, \label{eq:h_stacked_initial_condition}
            \end{align}
        \end{subequations}
        where
        \begin{align*}
            \hat{Q}_{t} &= \blkdiag(Q_{1,t}, \ldots, Q_{N,t}) 
                                    \in \RR^{\hat{n}\times\hat{n}},\\ 
            \hat{R}_{t} &= \blkdiag(R_{1,t}, \ldots, R_{N,t}) 
                                       \in \RR^{\hat{m}\times\hat{m}},\\ 
            \hat{B}_{t} &= \blkdiag(B_{1,t}, \ldots, B_{N,t})
                                       \in  \RR^{\hat{n}\times\hat{m}},\\ 
            \hat{M}_{t} &= \blkdiag(M_{1,t}, \ldots, M_{N,t}) 
                                       \in \RR^{\hat{n}\times\sum_{j\in\setN}(n_{0,j}+n_{K+1,j})},\\ 
            \hat{\boldsymbol{\alpha}}_{t} &= \col(\boldsymbol{\alpha}_{1,t},\boldsymbol{\alpha}_{N,t}) \in
                                        \RR^{\sum_{j\in\setN}(n_{0,j}+n_{K+1,j})},\\ 
            \hat{\boldsymbol{\beta}}_{t} &= \col(\underline{\beta}_{t},\ldots, \overline{\beta}_{t})
                                        \in\RR^{\bar{n}_{0}+\bar{n}_{N+1}},\\ 
            \hat{A}_{t} &= 
            \renewcommand{\arraystretch}{1,2}
            \setlength{\arraycolsep}{0.3pt}
            \begin{bmatrix}
            \bar{A}_{1,t}      &\bar{F}_{1,t}       &          &\\
            \bar{E}_{2,t} &\bar{A}_{2,t}     &\ddots          &\\
            &\ddots &\ddots    &\bar{F}_{N-1,t}\\
            & 		&\bar{E}_{N,t} &\bar{A}_{N,t}
            \end{bmatrix}, 
            \hat{N}_{t} = 
            \renewcommand{\arraystretch}{1.2}
            \setlength{\arraycolsep}{0.1pt}
            \left[
            \begin{array}{cc}
            \bar{E}_{1,t} 	 &0\\
            0            &\vdots\\
            \vdots       &0\\
            0 		 	 &\bar{F}_{N,t}
            \end{array}
            \right],
        \end{align*}
        and $\hat{N}_{t} \in \RR^{\hat{n}\times(\bar{n}_{0}+\bar{n}_{N+1})}$. Matrix $\hat{A}_{t} \in \RR^{\hat{n}\times\hat{n}}$ is a block tri-diagonal matrix of size $O(NK)$ and each of it's diagonal blocks are again block tri-diagonal of size $O(K)$.
        
        Finally, by defining $\tilde{u} = \col(\hat{u}_{0}, \ldots, \hat{u}_{T-1}) \in
        \RR^{\hat{m}T}$, and $\tilde{x} = \col(\hat{x}_{0}, \ldots, \hat{x}_{T}) \in
        \RR^{\hat{n}(T+1)}$,
        problem \eqref{eq:h_stacked_LQ_problem} can be reformulated by stacking across the temporal dimension of the problem:
        \begin{subequations}\label{eq:t_stacked_LQ_problem}
            \begin{equation}\label{eq:t_stacked_cost_function}
                \min_{\tilde{x},~\tilde{u}}
                \quad 
                \frac{1}{2}\biggl(
                \begin{bmatrix}
    					\tilde{x}\\
    					\tilde{u}
    				\end{bmatrix}^{\prime}
    				\begin{bmatrix}
    					\tilde{Q} & 0\\
    					0		   & \tilde{R}
    				\end{bmatrix}
    				\begin{bmatrix}
    					\tilde{x}\\
    					\tilde{u}
    				\end{bmatrix}\biggr)
            \end{equation}
            \text{subject to }
            \begin{equation}\label{eq:t_stacked_equality_constraint}
                \tilde{A}\tilde{x} + \tilde{B}\tilde{u} + \tilde{\omega} = 0 
            \end{equation}
        \end{subequations}
        where
        \begin{align*}
            \Tilde{Q} &= \blkdiag(\hat{Q}_{0}, \ldots, \hat{Q}_{T}) 
                                    \in \RR^{\hat{n}(T+1)\times\hat{n}(T+1)},\\ 
            \tilde{R} &= \blkdiag(\hat{R}_{0}, \ldots, \hat{R}_{T-1}) 
                                       \in \RR^{\hat{m}T\times\hat{m}T},\\ 
            \tilde{\omega} &= \col(0,\hat{\omega}_{0}, \ldots, \hat{\omega}_{T-1})
                                       \in  \RR^{\hat{n}(T+1)},\\ 
            \hat{\omega}_{0} &= \hat{M}_{t}\hat{\boldsymbol{\alpha}}_{t}
                                + \hat{N}_{t}\hat{\boldsymbol{\beta}}_{t} + \hat{\gamma} \in  \RR^{\hat{n}},\\
            \hat{\omega}_{t} &= \hat{M}_{t}\hat{\boldsymbol{\alpha}}_{t}
                                + \hat{N}_{t}\hat{\boldsymbol{\beta}}_{t}\in  \RR^{\hat{n}},~t\in\setT\backslash\{0,T\},\\
            \tilde{A} &= 
            \renewcommand{\arraystretch}{1,2}
            \setlength{\arraycolsep}{0.3pt}
            \begin{bmatrix}
            -I          &       &          &\\
            \hat{A}_{0} &-I     &          &\\
                        &\ddots &\ddots    &\\
                        & 		&\hat{A}_{T-1} &-I
            \end{bmatrix}, 
            \tilde{B} = 
            \renewcommand{\arraystretch}{1.2}
            \setlength{\arraycolsep}{0.1pt}
            \left[
            \begin{array}{ccc}
            0           &\hdots 	 &0\\
            \hat{B}_0   &\ddots      &\vdots\\
                        &\ddots      &0\\
                        & 		 	 &\hat{B}_{T-1}
            \end{array}
            \right],
        \end{align*}
        while $\tilde{B} \in \RR^{\hat{n}(T+1)\times\hat{m}T}$ and $\tilde{A}\in \RR^{\hat{n}(T+1)\times\hat{n}(T+1)}$ has a block bi-diagonal matrix due to the temporal coupling of variables in \eqref{eq:t_stacked_LQ_problem}.
        
    \subsection{KKT Conditions}\label{subsec:kkt_condtions}
        For problem \eqref{eq:h_stacked_LQ_problem}, the first-order optimality conditions, also known as KKT conditions, are given by
        \begin{subequations}\label{eq:kkt_conditions}
            \begin{align}
                \tilde{Q}\tilde{x} + \tilde{A}^{\prime}\tilde{\delta} &= 0,\\
                \tilde{R}\tilde{u} + \tilde{B}^{\prime}\tilde{\delta}  &= 0,\\
                \tilde{A}\tilde{x} + \tilde{B}\tilde{u} &= r_{\tilde{\delta}}.
            \end{align}
        \end{subequations}
        where $r_{\tilde{\delta}} = -\tilde{\omega}$.
        The variables $\tilde{\delta} = \col(\hat{\delta}_{0},\ldots,\hat{\delta}_{T}) \in \RR^{\hat{n}*(T+1)}$ are Lagrange multipliers. Since $Q$ and $R$ are positive definite, the problem \eqref{eq:t_stacked_LQ_problem} is strictly convex and the KKT conditions are necessary and sufficient for optimality \cite{Nocedal2000numerical}. Notice that, \eqref{eq:kkt_conditions} is a large-scale, sparse and structured linear system of equations whose size scales as $O(NKT)$.
        We can solve \eqref{eq:kkt_conditions} as
        \begin{subequations}
            \begin{align}
                \Delta \tilde{\delta} &= r_{\tilde{\delta}},\label{eq:block_tridiagonal_system}\\
                \tilde{x} &= \tilde{Q}^{-1}r_{\tilde{x}},\\
                \tilde{u} &= \tilde{R}^{-1}r_{\tilde{u}}.
            \end{align}
        \end{subequations}
        where $\Delta = (\tilde{A}\tilde{Q}^{-1}\tilde{A}^{\prime} + \tilde{B}\tilde{R}^{-1}\tilde{B}^{\prime})$, $r_{\tilde{x}} = (\tilde{A}^{\prime}\tilde{\delta})$ and $r_{\tilde{u}} = (\tilde{B}^{\prime}\tilde{\delta})$.
        
        Note that matrices $\tilde{Q} \succ 0$ and $\tilde{R} \succ 0$ are block diagonal with block size independent of $N,~K,$ and $T$. Therefore, computation of $\tilde{Q}^{-1}$ and $\tilde{R}^{-1}$ can be carried out efficiently. However, 
        solving \eqref{eq:block_tridiagonal_system} is computationally more expensive. Since $\tilde{Q}$ and $\tilde{R}$ are positive definite so their inverse are also positive definite, and $\tilde{A}$ is full row rank. Therefore, $\Delta$ is a positive definite and
        has a block tri-diagonal structure
        \begin{equation}\label{eq:structure_of_Delta}
			\Delta =
                    \renewcommand{\arraystretch}{1.2}
                    \setlength{\arraycolsep}{0.3pt}
                    \left[
    				\begin{array}{cccll}
    						\hat{\Psi}_0  &\hat{\Xi}_0^{\prime}   &                       &		             &\\
    						\hat{\Xi}_0   &\hat{\Psi}_1           &\hat{\Xi}_1            &		             &\\
    						               &\ddots                &\ddots	                &\ddots		       &\\
    						               &					 &\hat{\Xi}_{T-2}^{\prime}	&\hat{\Psi}_{T-1}	&\hat{\Xi}_{T-1}^{\prime}\\
    									   &					  &							&\hat{\Xi}_{T-1}    &\hat{\Psi}_T
					\end{array}
					\right] \in \RR^{\hat{n}(T+1) \times \hat{n}(T+1)},
		\end{equation}
        where 
		\begin{align*}
			\hat{\Psi}_t &\!=\!\hat{Q}_{t}^{-1}\!+\!\hat{A}_{t\!-\!1}\hat{Q}_{t\!-\!1}^{-1}\hat{A}_{t\!-\!1}^{\prime}\!+\!\hat{B}_{t\!-\!1}\hat{R}_{t\!-\!1}^{-1}\hat{B}_{t\!-\!1}^{\prime}~t\in\setT\backslash\{0\},\\
			\hat{\Xi}_{t} &\!=\! -\hat{A}_t\hat{Q}_{t}^{-1} \in \RR^{\hat{n} \times \hat{n}},~ t\in\setT\backslash\{T\},
		\end{align*}
        and $\hat{\Psi}_0 = \hat{Q}_{0}^{-1} \in \RR^{\hat{n} \times \hat{n}}$. 
        Note that each $\hat{\Psi}_t\in \RR^{\hat{n} \times \hat{n}},~t\in\setT\backslash \{0\}$ has a block penta-diagonal structure because $\hat{A}_t$ are block tri-diagonal of size $O(NK)$. Let $\setV = \{1,2,\ldots,V\}$, with $V = N/2$ when $N$ is even and $V = (N+1)/2$ otherwise. Also define,
        \begin{subequations}
    		\label{eq:constructing_Phi_Omega_matrices} 
    		\begin{align}\label{eq:structure_Phi_2x2_a}
    			\bar{\Phi}_{v,t} &= 
    			\begin{bmatrix}
    				\bar{Z}_{2v - 1,t} &\bar{Y}_{2v,t}^{\prime}\\
    				\bar{Y}_{2v,t}	   &\bar{Z}_{2v,t}
    			\end{bmatrix},~ v \in \mathcal{V}\backslash\{V\},\\ 
    			\bar{\Phi}_{V,t} &=	
    			\begin{cases}
    				\begin{bmatrix}
    					\bar{Z}_{N-1,t} &\bar{Y}_{N,t}^{\prime}\\
    					\bar{Y}_{N,t}	 &\bar{Z}_{N,t}
    				\end{bmatrix}, & \text{$V$ even},\\
    				\bar{Z}_{N,t}, & \text{$V$ odd},
    			\end{cases}\\
    			\bar{\Omega}_{v,t} &= 
    			\begin{bmatrix}
    				\bar{W}_{2v - 1,t} &\bar{Y}_{2v - 1,t}\\
    				0		     &\bar{W}_{2v}
    			\end{bmatrix},~ v \in
    			\mathcal{V}\backslash\{1,V\}, \\
    			\bar{\Omega}_{V,t} &=	
    			\begin{cases}
    				\begin{bmatrix}
    					\bar{W}_{N-1,t} &\bar{Y}_{N-1,t}\\
    					0		 &\bar{W}_{N,t}
    				\end{bmatrix}, & \text{$N$ even},\\ 
    				\begin{bmatrix}
    					\bar{V}_{N,t} &\bar{Y}_{N,t}
    				\end{bmatrix}, & \text{$N$ odd},
    			\end{cases}
    		\end{align}
    	\end{subequations}
        where 
    	\begin{subequations}\label{eq:blocks_of_Psi_defined}
    		\begin{align*}
    			\bar{Z}_{j,t} &\!=\!  \bar{A}_{j,t}\bar{Q}_{j,t}^{-1}\bar{A}_{j,t}^{\prime}\!+\! \bar{E}_{j,t}\bar{Q}_{j-1,t}^{-1}\bar{E}_{j,t}^{\prime}\!+\! \bar{F}_{j,t}\bar{Q}_{j+1,t}^{-1}\bar{F}_{j,t}^{\prime},\\
    			\bar{Y}_{j,t} &\!=\! \bar{E}_{j,t}\bar{Q}_{j-1,t}^{-1}\bar{A}_{j-1,t}^{\prime}\!+\! \bar{A}_{j,t}\bar{Q}_{j,t}^{-1}\bar{F}_{j-1,t}^{\prime},\\
    			\bar{V}_{j,t} &\!=\! \bar{E}_{j,t}\bar{Q}_{j-1,t}^{-1}\bar{F}_{j-2,t}^{\prime},
    		\end{align*}
    	\end{subequations}
        for $j\in\setN$, $t\in\setT\backslash\{T\}$, with $\bar{Q}_{0,t} = 0$, $\bar{Q}_{N+1,t} = 0$ and $\bar{F}_{0,t} = 0$. Given this, $\hat{\Psi}_{t}$ for $t\in\setT\backslash\{T\}$ becomes a block tri-diagonal matrix given by
        \begin{equation}\label{eq:structure_of_Psi_t}
			\hat{\Psi}_{t} = \hat{\Phi}_{t} - \hat{\Omega}_{t} = 
                    \renewcommand{\arraystretch}{1.2}
                    \setlength{\arraycolsep}{0.3pt}
                    \left[
    				\begin{array}{cccll}
    						\bar{\Phi}_{1,t}  &\bar{\Omega}_{2,t}^{\prime}   &                       &		             &\\
    						\bar{\Omega}_{2,t}   &\bar{\Phi}_{2,t}           &\bar{\Omega}_{3,t}            &		             &\\
    						               &\ddots                &\ddots	                &\ddots		       &\\
    						               &					 &\bar{\Omega}_{V-1,t}^{\prime}	&\bar{\Phi}_{V-1,t}	&\bar{\Omega}_{V,t}^{\prime}\\
    									   &					  &							&\bar{\Omega}_{V,t}    &\bar{\Phi}_{V,t}
					\end{array}
					\right],
		\end{equation}
        where $\hat{\Phi}_{t} = \blkdiag(\bar{\Phi}_{1,t},\ldots,\bar{\Phi}_{V,t})$ and $\hat{\Omega}_{t} = \hat{\Phi}_{t} - \hat{\Psi}_{t}$. This structured splitting of $\hat{\Psi}$ will be needed in the next section.
        
        In the next section, we propose an iterative algorithm based on preconditioned conjugate gradient method (PCGM) to solve \eqref{eq:block_tridiagonal_system}.
        The main contribution of this paper is the development of a structured preconditioner based on nested block Jacobi iterations whose properties are detailed in Section~\ref{sec:nested_jacobi_preconditioner}.
                

\section{Preconditioned Conjugate Gradient Method}\label{sec:pcgm}
    The positive definite linear system of equations \eqref{eq:block_tridiagonal_system} can be solved using conjugate gradient method (CGM) \cite{Hestenes1952methods}. Let $e^{(k)} = \tilde{\delta}^{(k)} - \tilde{\delta}^{*}$ be the error between $k$-th iterate $\tilde{\delta}^{(k)}$ of the CGM and the exact solution $\tilde{\delta}^*$ of \eqref{eq:block_tridiagonal_system}. It can be shown that $e^{(k)}$ satisfies the
	following~\cite[Thm.~6.29]{Saad2003iterative}:  
	\begin{equation}\label{eq:convergence_rate_CG_method}
		\|e^{(k)}\|_{\Psi} \leq
		2\left((\sqrt{\kappa(\Delta)}
		-1)\big/ (\sqrt{\kappa(\Delta)} +
		1)\right)^{i}\|e^{(0)}\|_{\Delta}, 
	\end{equation}
	where $\|e\|_{\Delta} = e^{\prime}\Delta e$,
	$\kappa(\Delta) = \lambda_{\mathrm{max}}(\Delta)\big/
	\lambda_{\mathrm{min}}(\Delta)$ is the condition number and $\lambda_{\mathrm{max}}(\Delta)$, $\lambda_{\mathrm{min}}(\Delta)$) are the maximum and minimum eigenvalue of $\Delta$, respectively. It shows that the CGM converges faster if $\kappa(\Delta)$ is closer to $1$. Therefore, to improve the performance of CGM, the system of equations \eqref{eq:block_tridiagonal_system} is transformed such that the condition number of the transformed coefficient matrix is improved. The transformed system of equations is known as preconditioned system. Let $P = P^{\prime}$ be a positive definite matrix, then solving \eqref{eq:block_tridiagonal_system} is equivalent to solving the transformed system
    \begin{equation}\label{eq:pcg_system}
		P^{-1/2}\Delta P^{-1/2} \tilde{\gamma} = r_{\tilde{\gamma}}, 
	\end{equation}
	where $\tilde{\gamma} = P^{1/2}\tilde{\delta}$ and $r_{\tilde{\gamma}} = P^{-1/2}r_{\tilde{\delta}}$. An efficient implementation of preconditioned CGM (PCGM) is given in Algorithm~\ref{alg:pcg} from~\cite{Hackbusch2016iterative}. 
    \begin{algorithm}[t]
		\caption{PCGM~\cite{Hackbusch2016iterative}
			for
			\eqref{eq:block_tridiagonal_system}
			with preconditioner
			$P$.}  
		\label{alg:pcg}
		\begin{algorithmic}[1] 
			\State \textbf{initialize}
			$\tilde{\delta}^{(0)}$, $\epsilon$,
			$\text{iter}_{\text{max}}$ 
			\State $r^{(0)} = r_{\tilde{\theta}} - \Delta
			\tilde{\delta}^{(0)}$ 
			\State \textbf{solve} $P d^{(0)} =
			r^{(0)}$ \label{step:precond0}
			\State $\mu^{(0)} = 
			(d^{(0)})^\prime r^{(0)} $  
			\State $i = 0$ \label{step:initialize}
			\While {$i < \text{iter}_{\text{max}}$}
			\State $y^{(i)} = \Delta
			{d}^{(i)}$ \label{step:matrix_vector_product} 
			\State $\vartheta^{(i)} =
			{\mu^{(i)}}/{(
				(y^{(i)})^\prime d^{(i)})
			}$ \label{step:alpha_update} 
			\State $\tilde{\delta}^{(i+1)} =
			\tilde{\delta}^{(i)} 
			+ \vartheta^{(i)}
			d^{(i)}$\label{step:variable_update} 
			\State $r^{(i+1)} = r^{(i)} -
			\vartheta^{(i)}
			y^{(i)}$\label{step:residual_update} 
			\State \textbf{if $\|r^{(i+1)}
				\|_{\infty} < \epsilon $
				exit}  \label{step:error_update} 
			\State \textbf{solve} $P q^{(i+1)} =
			r^{(i+1)}$ \label{step:preconditioning} 
			\State $\mu^{(i+1)} = 
			(q^{(i+1)})^\prime r^{(i+1)}
			$  \label{step:beta_update} 
			\State $d^{(i+1)} = r^{(i+1)} +
			\left({\mu^{(i+1)}}/{\mu^{(i)}}\right)
			d^{(i)}$ \label{step:direction_update} 
			\State $i = i+1$
			\EndWhile
		\end{algorithmic}
	\end{algorithm}
    Note that lines~\ref{step:matrix_vector_product} to~\ref{step:direction_update} constitute one PCG step. The transformed system $P^{-1/2}\Delta P^{-1/2}$ is never formed explicitly. Instead, the preconditioning is performed by solving a linear system of equations at lines~\ref{step:precond0} and~\ref{step:preconditioning} of Algorithm~\ref{alg:pcg}. Choice of preconditioner $P$ determines the overall complexity of PCGM. 
    Ideally, if we use the preconditioner $P = \Delta$, this will result in $P^{-1/2}\Delta P^{-1/2} = I$. However, the linear system to be solved at the preconditioning lines~\ref{step:precond0} and~\ref{step:preconditioning} becomes the original problem. The incomplete sparse LU factorization of the coefficient matrix $\Delta$ can be used as a preconditioner \cite{Benzi2003robust, Xia2017effective}.  However, to ensure that the resulting preconditioner is both positive definite and effective, it may be necessary to use incomplete LU factors that are denser (i.e., have less structure) than $\Delta$. 
    
    In the next section, we devise a structured preconditioning approach based on nested block Jacobi iterations. We apply a fixed number of inner iterations at each outer iteration of block Jacobi method to approximately implement lines~\ref{step:precond0} and~\ref{step:preconditioning} with $P = \Delta$. The proposed approach is inspired by the ideas from \cite{Nichols1973convergence, Johnson1983polynomial, Lanzkron1990convergence, Migallon1997convergence, Cao1999two, Cao2000convergence}. It extends the developments presented in \cite{Zafar2022structured} where same ideas have already been explored
    within the context of 1D chains.
    
\section{Nested Block Jacobi Preconditioner}\label{sec:nested_jacobi_preconditioner}
    Let $\Delta = \tilde{\Psi} - \tilde{\Xi} \succ 0$ is a matrix splitting such that
    \begin{subequations}\label{eq:splitting_Delta}
        \begin{align}
            \tilde{\Psi} &= \blkdiag(\hat{\Psi}_{0}, \ldots, \hat{\Psi}_{T}) \succ 0,\\
            \tilde{\Xi} &= \tilde{\Psi} - \Delta.
        \end{align}
    \end{subequations}
    Then, the block Jacobi method for solving \eqref{eq:block_tridiagonal_system} involves the following:
    \begin{equation}\label{eq:outer_jacobi_iterations}
        \tilde{\Psi}\tilde{\delta}^{(s+1)} = r_{\tilde{\delta}} + \Xi\tilde{\delta}^{(s)},
    \end{equation}
    where $s$ is the number of block Jacobi iterations. These iterations converge
    to the solution of \eqref{eq:block_tridiagonal_system} if and only if
    \begin{equation}\label{eq:convergence_condition}
          \rho(\Psi^{-1}\Xi) < 1,
    \end{equation}
    where $\rho(\cdot)$ denotes spectral radius~\cite[Thm. 2.16]{Hackbusch2016iterative}. Given that $\Delta$ is a positive definite block tri-diagonal matrix, and the splitting \eqref{eq:splitting_Delta}, condition \eqref{eq:convergence_condition} always holds when \eqref{eq:outer_jacobi_iterations} is solved exactly~\cite[Lem. 4.7 with Thm. 4.18]{Hackbusch2016iterative}. However, instead of solving \eqref{eq:outer_jacobi_iterations} exactly, we propose a fixed number of inner iterations. It is apparent from \eqref{eq:structure_of_Psi_t}, that $\tilde{\Psi} \succ 0$ has a block tri-diagonal structure with $\hat{\Omega}_{1,t} = 0$ and $\hat{\Omega}_{V+1,t} = 0$ for $t \in \setT$. Let $\tilde{\Psi} = \tilde{\Phi} - \tilde{\Omega} \succ 0$ such that 
    \begin{subequations}\label{eq:splitting_Psi}
        \begin{align}
            \tilde{\Phi} &= \blkdiag(\hat{\Phi}_{0}, \ldots, \hat{\Phi}_{T}) \succ 0,\\
            \tilde{\Omega} &= \tilde{\Phi} - \tilde{\Psi}.
        \end{align}
    \end{subequations}
    Then for each outer iteration $s$, we propose $L$ number of inner iterations. This results in a nested block Jacobi method (NBJM) given in Algorithm~\ref{alg:NBJM}.
    %
    %
    \begin{algorithm}[t]
		\caption{NBJM for \eqref{eq:block_tridiagonal_system}}  
		\label{alg:NBJM}
		\begin{algorithmic}[1] 
			\State \textbf{initialize}
			$\tilde{\delta}^{(0)} = 0$
            \For {$s = 0,\ldots,S-1$}
                \State \textbf{set } $r_{\tilde{\theta}}^{(s)} = r_{\tilde{\delta}} + \Xi\tilde{\delta}^{(s)}$ \textbf{ and initialize } $\tilde{\theta}^{(0)} = 0$ ,
                \For {$l = 0,\ldots,L-1$}\label{stepNBJM:matrix_vector_product}
                    \State \textbf{solve} $\tilde{\Phi} \tilde{\theta}^{(l+1)} = r_{\tilde{\theta}}^{(s)} + \tilde{\Omega}\tilde{\theta}^{(l)}$ \label{stepNBJM:inner_iteration}
                \EndFor
                \State \textbf{if $\|\tilde{\delta}^{(s+1)} - \tilde{\delta}^{(s)}\|_{\infty} < \epsilon$ exit}
            \EndFor
		\end{algorithmic}
	\end{algorithm}
    %
    Note that for the proposed splitting \eqref{eq:splitting_Psi} of the positive definite block tri-diagonal matrix $\tilde{\Psi}$, condition $\rho(\tilde{\Phi}^{-1}\tilde{\Omega}) < 1$ always holds~\cite[Lem. 4.7 with Thm. 4.18]{Hackbusch2016iterative}. 
    %
    %
    %
    \subsection{Positive-definiteness of Preconditioner}
        We propose to apply fixed number of inner-outer iterations of NBJM (i.e., Algorithm~\ref{alg:NBJM}), with $P = \Delta$ at lines~\ref{step:precond0} and~\ref{step:preconditioning} in Algorithm~\ref{alg:pcg}. Executing this is equivalent to the use of a positive definite preconditioner. First, we present the following technical lemma which leads to the main result.
        \begin{lemma}\label{lem:inverse_of_Psi}
            Given that for the proposed splitting \eqref{eq:splitting_Psi} of the positive definite block tri-diagonal matrix $\tilde{\Psi}$, condition $\rho(\tilde{\Phi}^{-1}\tilde{\Omega}) < 1$ always holds. Let $L$ is a positive even integer i.e., $L \in \{2,4,\ldots\}$, the inverse of $\tilde{\Psi}$ can be expressed as
            \begin{equation}\label{eq:inverse_of_Psi}
                \tilde{\Psi}^{-1} = \tilde{\Upsilon}_{L} + \tilde{\Pi}_{L},
            \end{equation}
            where $\tilde{\Upsilon}_{L} = \sum_{l=0}^{L-1}(\tilde{\Phi}^{-1}\tilde{\Omega})^{l}\tilde{\Phi}^{-1} \succ 0$ and $\tilde{\Pi}_{L} = \sum_{l=L}^{\infty}(\tilde{\Phi}^{-1}\tilde{\Omega})^{l}\tilde{\Phi}^{-1} \succeq 0$.
        \end{lemma}
        \begin{proof}
            First note that for the proposed splitting \eqref{eq:splitting_Psi} of $\tilde{\Psi} = \tilde{\Phi} - \tilde{\Omega} \succ 0$, with $\Lambda = \blkdiag(\Lambda_{T+1},\ldots,\Lambda_1)$ and $\Lambda_t = (-I)^{t}$ for $t = 1,\ldots,T+1$, it holds that $\tilde{\Phi} + \tilde{\Omega} = \Lambda^{\prime}\tilde{\Psi}\Lambda \succ 0$. Then using $(\tilde{\Phi}^{-1}\tilde{\Omega})^{l}\tilde{\Phi}^{-1} = \tilde{\Phi}^{-1}(\tilde{\Omega}\tilde{\Phi}^{-1})^{l}$, note that
            \begin{align*}
                \tilde{\Upsilon}_{L} &= \sum_{l=0}^{(L/2)-1}(\tilde{\Phi}^{-1}\tilde{\Omega})^{l}\tilde{\Phi}^{-1}(\tilde{\Phi} + \tilde{\Omega})\tilde{\Phi}^{-1}(\tilde{\Omega}\tilde{\Phi}^{-1})^{l}\\
                &= \sum_{l=1}^{(L/2)-1}(\tilde{\Phi}^{-1}\tilde{\Omega})^{l}\tilde{\Phi}^{-1}(\tilde{\Phi} + \tilde{\Omega})\tilde{\Phi}^{-1}(\tilde{\Omega}\tilde{\Phi}^{-1})^{l}\\
                &+ \tilde{\Phi}^{-1}(\tilde{\Phi} + \tilde{\Omega})\tilde{\Phi}^{-1},
            \end{align*}
            and
            \begin{align*}
                \tilde{\Pi}_{L} &= \sum_{l=(L/2)}^{\infty}(\tilde{\Phi}^{-1}\tilde{\Omega})^{l}\tilde{\Phi}^{-1}(\tilde{\Phi} + \tilde{\Omega})\tilde{\Phi}^{-1}(\tilde{\Omega}\tilde{\Phi}^{-1})^{l},
            \end{align*}
            for $L \in \{2,4,\ldots\}$. Note that $\tilde{\Phi}^{-1}(\tilde{\Phi} + \tilde{\Omega})\tilde{\Phi}^{-1} \succ 0$ and $(\tilde{\Phi}^{-1}\tilde{\Omega})^{l}\tilde{\Phi}^{-1}(\tilde{\Phi} + \tilde{\Omega})\tilde{\Phi}^{-1}(\tilde{\Omega}\tilde{\Phi}^{-1})^{l} \succeq 0$ for all $l\in \NN$. Therefore, $\tilde{\Upsilon}_{L} \succ 0$ and $\tilde{\Pi}_{L} \succeq 0$ as claimed.
        \end{proof}
        \begin{remark}\label{cor:inverse_of_Psi}
            Note that from Lem.~\ref{lem:inverse_of_Psi}, $\tilde{\Psi}^{-1} - \tilde{\Upsilon}_L = \tilde{\Pi}_{L} \succeq 0$ which implies that $\tilde{\Upsilon}_L^{-1} - \tilde{\Psi} \succeq 0$ \cite[Cor. 7.7.4]{Horn2013matrix}.
            Also note that for the proposed splitting \eqref{eq:splitting_Delta} of $\Delta \succ 0$ in \eqref{eq:structure_of_Delta}, with $\Lambda = \blkdiag(\Lambda_{T+1},\ldots,\Lambda_1)$ and $\Lambda_t = (-I)^{t}$ for $t = 1,\ldots,T+1$, it holds that $\tilde{\Psi} + \tilde{\Xi} = \Lambda^{\prime}\tilde{\Psi}\Lambda \succ 0$. 
            Therefore, $(\tilde{\Upsilon}_L^{-1} + \tilde{\Xi}) - (\tilde{\Psi} + \tilde{\Xi}) \succeq 0 \implies (\tilde{\Upsilon}_L^{-1} + \tilde{\Xi}) \succeq (\tilde{\Psi} + \tilde{\Xi}) \succ 0$.
        \end{remark}
        Now we present the main result.
        \begin{theorem}\label{thm:posdef_preconditioner}
            Given $S \in \NN$ and $L \in \{2,4,\ldots\}$, the $S^{\mathrm{th}}$ iterate of Algorithm~\ref{alg:NBJM} satisfies $P_{S} \tilde{\delta}^{S} = r_{\tilde{\delta}}$ with $P_{S} = \varDelta_{S}^{-1} \succ 0$, where $\varDelta_{S} = \sum_{s=0}^{S-1}(\tilde{\Upsilon}_{L}\tilde{\Xi})^{s}\tilde{\Upsilon}_{L}$ and $\tilde{\Upsilon}_{L} = \sum_{l=0}^{L-1}(\tilde{\Phi}^{-1}\tilde{\Omega})^{l}\tilde{\Phi}^{-1}$.
        \end{theorem}
        \begin{proof}
            Noting that for each outer iteration $s$ of Algorithm~\ref{alg:NBJM}, we execute $L$ number of inner iterations with $\tilde{\theta}^{(0)} = 0$. Therefore, $\tilde{\delta}^{(s)} = r_{\tilde{\delta}} + (\tilde{\Upsilon}_{L}\tilde{\Xi})\tilde{\delta}^{(s-1)} = \varDelta_{S}r_{\tilde{\delta}} + (\tilde{\Upsilon}_{L}\tilde{\Xi})^{S}\tilde{\delta}^{(0)} = \varDelta_{S}r_{\tilde{\delta}}$. 
            Then note that
            $\varDelta_1 = \tilde{\Upsilon}_1 \succ 0$, $\tilde{\Upsilon}_{L} (\tilde{\Upsilon}_{L}^{-1} + \Xi)\tilde{\Upsilon}_{L} \succ 0$ and using $(\tilde{\Upsilon}_{L}\tilde{\Xi})^{s}\tilde{\Upsilon}_{L} = \tilde{\Upsilon}_{L}(\tilde{\Xi}\tilde{\Upsilon}_{L})^{s}$ that
            \begin{align*}
    			\varDelta_{2S} &= \sum_{s=0}^{S-1}
    			(\tilde{\Upsilon}_{L}\tilde{\Xi})^{s} \tilde{\Upsilon}_{L}
    			(\tilde{\Upsilon}_{L}^{-1} + \Xi) \tilde{\Upsilon}_{L} (\tilde{\Xi}
    			\tilde{\Upsilon}_{L})^{s}  \\ 
    			&= \sum_{s=1}^{S-1}  (\tilde{\Upsilon}_{L}\tilde{\Xi})^{s}
    			\tilde{\Upsilon}_{L} (\tilde{\Upsilon}_{L}^{-1} + \Xi) \tilde{\Upsilon}_{L} (\tilde{\Xi}
    			\tilde{\Upsilon}_{L})^{s} \\ 
    			&\quad + \tilde{\Upsilon}_{L} (\tilde{\Upsilon}_{L}^{-1} + \Xi)
    			\tilde{\Upsilon}_{L} \succ 0, 
    		\end{align*}
    		and
    		\begin{align*}
    			\varDelta_{2S+1} &= \sum_{s=0}^{2S-1}
    			(\tilde{\Upsilon}_{L}\tilde{\Xi})^{s}\tilde{\Upsilon}_{L} +
    			(\tilde{\Upsilon}_{L}\tilde{\Xi})^{2{S}}\tilde{\Upsilon}_{L}  
    			\\ 
    			&= 
    			\varDelta_{2S} + (\tilde{\Upsilon}_{L}\tilde{\Xi})^{S} \tilde{\Upsilon}_{L}
    			((\tilde{\Upsilon}_{L}\tilde{\Xi})^{S})^{\prime}\succ 0,
    		\end{align*}
    		for ${S}\in\mathbb{N}$. Therefore, $\varDelta_{S}\succ 0$ and consequently $P_{S} = \varDelta_{S}^{-1} \succ 0$..
        \end{proof}

    \subsection{Decomposable Computations}\label{subsec:decomposable_computations}
        Note that we do not explicitly construct the preconditioner $P_{S}$. Instead, at each PCG step, we execute $S$ number of outer iterations of Algorithm~\ref{alg:NBJM} such that for each outer iteration $s$, $L$ number of inner iterations are performed. At each inner iteration, we need to solve
        \begin{equation}\label{eq:inner_iteration_NBJM}
            \tilde{\Phi} \tilde{\theta}^{(l+1)} = r_{\tilde{\theta}}^{(s)} + \tilde{\Omega}\tilde{\theta}^{(l)}.
        \end{equation}
        Let, $\tilde{\theta} = \col(\hat{\theta}_{0}, \ldots, \hat{\theta}_{T})$ where each $\hat{\theta}_{t} = \blkdiag(\bar{\theta}_{1,t}, \ldots, \bar{\theta}_{V,T})$ for $t\in \setT$. Then, the computations in \eqref{eq:inner_iteration_NBJM} can be decomposed into $V\times (T+1)$ smaller problems
        \begin{equation}\label{eq:decomposed_problem}
            \bar{\Phi}_{v,t} \bar{\theta}_{v,t}^{(l+1)} = \bar{\tau}_{v,t},
        \end{equation}
        for $v\in \setV$ and $t\in\setT$, where $\bar{\tau}_{v,t} = (r_{\tilde{\theta}}^{(s)})_{v,t} + \bar{\Omega}_{v,t}\bar{\theta}_{v-1,t}^{(l)} + \bar{\Omega}_{v+1,t}^{\prime}\bar{\theta}_{v+1,t}^{(l)}$, and $(r_{\tilde{\theta}}^{(s)})_{v,t}$ represents the decomposition of vector $r_{\tilde{\theta}}^{(s)}$, accordingly. Note that for each smaller \eqref{eq:decomposed_problem}, the coefficient matrix $\bar{\Phi}_{v,t}$ is a block banded matrix of size $O(K)$, where the subblocks are of size $O(n_{i,j})$ independent of $K$, $N$ and $T$. Therefore for each $v\in\setV$ and $t\in setT$, the linear system \eqref{eq:decomposed_problem} can be solved using block Cholesky factorization of $\bar{\Phi}_{v,t}$ with arithmetic complexity of $O(K)$. Note that this factorization is performed once at the start of Algorithm~\ref{alg:pcg}. 
        Therefore, the computations at preconditioning lines~\ref{step:precond0} and~\ref{step:preconditioning} decompose into $V\times (T+1)$ parallel threads with per thread arithmetic complexity of $O(K)$. By defining $\tilde{\delta} = \col(\hat{\delta}_{0}, \ldots, \hat{\delta}_{T})$ where each $\hat{\delta}_{t} = \blkdiag(\bar{\delta}_{1,t}, \ldots, \bar{\delta}_{V,T})$ and for $t\in \setT$ and partitioning $\Delta$ and $\tilde{\Xi}$ accordingly, the matrix vector products at line~\ref{step:matrix_vector_product} of Algorithm~\ref{alg:pcg} 
        can also be decomposed into $V\times (T+1)$ parallel threads connected in a 2D grid structure with localized data exchange.
        In Table~\ref{tab:PCG_complexity_analysis}, we present the arithmetic complexity of each line of Algorithm~\ref{alg:pcg}. It shows the overall arithmetic complexity as well as per-thread, along with the overhead of scalar-value data exchange for each thread in case of parallel implementation on $V\times(T+1)$ processors. The arithmetic complexity of each PCG step is dominated by the preconditioning line~\ref{step:preconditioning}. With $L$ and $S$ being fixed, the overall arithmetic complexity of each PCG step becomes $O(KNT)$. Note that the computations involved to form the dot products at lines~\ref{step:alpha_update} and~\ref{step:beta_update} and to update the infinity norm of the residual at line~\ref{step:error_update} are sequential in nature. In case of parallel implementation, they can be carried out by performing a backward-forward across the 2D grid with localized scalar data exchange.
        It should be noted that the total arithmetic complexity of PCGM to converge to a solution of desired accuracy depends on the total number of steps performed. In worst case, the number of steps can be of $O(KNT)$. This will result in $O((KNT)^2)$ arithmetic complexity. However, it often performs much better than this worst case complexity, as illustrated in the next section.
        
        \renewcommand{\arraystretch}{1.5}
	\begin{table}[t]
		\begin{tabular}{l|l|ll}
			\hline
			\rowcolor[HTML]{b0f193} 
			&
			\textbf{\begin{tabular}[c]{@{}l@{}}Single
					Thread \end{tabular}} &
			\multicolumn{2}{l}{\cellcolor[HTML]{b0f193}\textbf{\begin{tabular}[c]{@{}l@{}} $V\times~(T+1)$ Parallel Threads\end{tabular}}} \\ \hline
			\rowcolor[HTML]{dcf8fb} 
			\textbf{Alg.~\ref{alg:pcg}}
			& \textbf{Computations}
			& \textbf{\begin{tabular}{l}
					\!\!\!\!Computations\\[-2pt]
					\!\!\!\!\!\! per
					thread\end{tabular}}
			&
			\textbf{\begin{tabular}{l}\!\!\!\!
					Data
					xchg.\\[-2pt]  
					\!\!\!\! per
					thread \end{tabular}}
			\\ \hline  
			Line~\ref{step:matrix_vector_product}:	& $O(KNT\check{n}^{3})$           &$O(K\check{n}^{3})$        &$O(K\check{n})$       \\ \hline
			Line~\ref{step:alpha_update}:    		& $O(KNT\check{n}^{3})$          &$O(K\check{n}^{3})$        &$O(1)$                \\ \hline
			Line~\ref{step:variable_update}:        & $O(KNT\check{n}^{3})$          &$O(K\check{n}^{3})$        &0                    \\ \hline
			Line~\ref{step:residual_update}:        & $O(KNT\check{n}^{3})$          &$O(K\check{n}^{3})$        &0                    \\ \hline
			Line~\ref{step:error_update}:			& $O(KNT\check{n}^{3})$          &$O(K\check{n}^{3})$ 	   &$O(1)$              \\ \hline
			Line~\ref{step:preconditioning}:		& $O(LSKNT\check{n}^{3})$        &$O(K\check{n}^{3})$        &$O(LSK\check{n})$   \\ \hline
			Line~\ref{step:beta_update}:			& $O(KNT\check{n}^{3})$          &$O(K\check{n}^{3})$        &$O(1)$                \\ \hline
			Line~\ref{step:direction_update}:    	& $O(KNT\check{n}^{3})$          &$O(K\check{n}^{3})$        &0   	              \\ \hline
		\end{tabular}\\
		\caption{Arithmetic complexity and data-exchange overhead of the main loop of Algorithm~\ref{alg:pcg} where $\check{n} =\max_{i,j}(n_{i,j})$, where $n_{i,j}$ is the size of the state $x_{i,j}$; and $L$, $S$ are the fixed number of inner-outer iterations of Algorithm~\ref{alg:NBJM}.}
		\label{tab:PCG_complexity_analysis}
	\end{table}

\section{Numerical Experiments}\label{sec:numerical_experiments}
    In this section, we present results of the numerical experiments for solving LQ optimal control for two cases.
    
    \textbf{Case (I):} We consider a 2D grid of mass-spring-damper systems with $N\times K$ masses, where each sub-system $(i,j)$ has the discrete-time state space dynamics of the form \eqref{eq:2d_graph_dynamics} with four states and two inputs. The cost function has $Q_{i,j,t} = I \in \RR^{4\times 4}$ and $R_{i,j,t} = 2$ for $i \in \mathcal{K}$, $j \in \mathcal{N}$, $t \in \mathcal{T}$. The mass, spring constant, and damping coefficient parameters are randomly selected between $0.8$ and $1.5$ to create a heterogeneous network.
    
    \textbf{Case (II):} We consider an LQ optimal control problem of an automated irrigation network under distant-downstream control \cite{Cantoni2007control} shown in Fig.~\ref{fig:water_irigation_network}. The network consists of $N$ controlled pools in the primary channel and $K$ controlled pools in each secondary channel originating from the primary channel's pools. The dynamics of each controlled pool can be described by a sub-system in the form of \eqref{eq:2d_graph_dynamics} with $F_{i,j,t} = 0$ and $G_{i,j,t} = 0$. That is, there is a direct coupling between the states of sub-systems $(i,j)$ and $i,j-1$ in the primary channel and $(i,j)$ and $i-1,j$ in the secondary channels. However, there is no inter-coupling between the states of subsystems in the secondary channels. The control inputs $u_{i,j,t}$ are the adjustment of the water-level reference, and the states $x_{i,j,t}$ are the deviation of state trajectory from equilibrium. For each sub-system $i\in\mathcal{K},~j \in\mathcal{N}$, the number of states $n_{i,j} = 4$, and the number of inputs $m_{i,j} = 1$. The corresponding cost function has $Q_{i,j} = I \in \RR^{4\times 4}$ and $R_{i,j} = 1$.

    The LQ optimal control problem for both cases is solved by solving the corresponding KKT conditions of the form \eqref{eq:kkt_conditions}. We compare results of the following three methods for solving \eqref{eq:block_tridiagonal_system}. 
    \begin{itemize}
        \item The proposed PCGM in Algorithm~\ref{alg:pcg}, with proposed nested block Jacobi preconditioning with $L=2$ and $S=2$.
        \item A standard nested block Jacobi method (see e.g. \cite{Cao2000convergence}).
        \item MATLAB's Backslash (CHOLMOD \cite{Chen2008cholmod}) with sparse matrix representation
    \end{itemize}
    The experiment are conducted with $N = K = T$ varied from $10$ to $100$ such that the number of scalar variables in the largest problems for both cases are of $O(10^6)$.
    We consider the processor time of a single thread implementation which serves as a proxy for the overall arithmetic complexity of each method. We set the stopping criteria for the infinity norm of the residuals of the PCGM and standard nested block Jacobi method to $\epsilon = 10^{-9}$. Note that for solving \eqref{eq:block_tridiagonal_system} with the standard nested block Jacobi method the outer iterations are performed until stopping criteria is met, while the inner iterations are fixed to $L=2$. All computations are performed on a Windows PC with Intel Core-i7 processor with $2.4GHz$ clock speed and 16GB of RAM.

    \textbf{Case (I): 2D Grid of spring-mass-damper network.} Fig.~\ref{fig:SMD_PCG_J_Iters} shows the total number of steps of the proposed PCGM with $L=2,~S=2$ and total number of outer steps of the standard NBJM with $L=2$ against the total number of variables in \eqref{eq:block_tridiagonal_system}. It shows that the standard NBJM consistently takes a large number of steps, in the order of thousands. On the other hand, our proposed PCGM consistently takes fewer number of steps lesser than hundred as compared to the size of the system. For instance, for $K = N = T = 40$, the total number of unknown are $O(10^5)$ while the total number of PCGM steps performed are $81$ only. Such small number of steps of proposed PCGM are particularly favourable for a distributed implementation resulting in a small data exchange overhead.
    
    Fig.~\ref{fig:SMD_log_Time_Normalized} shows the normalized processor time for solving \eqref{eq:block_tridiagonal_system} using all three approaches. It shows that along the line $K=N=T$, the per step arithmetic complexity of PCGM and standard NBJM scales as $O(K^{3})$. However, due to the large number of step performed by NBJM, the overall processor time required is an order larger than the proposed PCGM. In contrast, for smaller systems MATLAB's Backslash performs well but its overall processor time scales as $O(K^6)$. This is because Backslash invokes sparse cholesky factorization from CHOLMOD \cite{Chen2008cholmod}. Despite its superior capabilities including permutation of variables for computational benefit, the number of fill-ins during the factorization increases exponentially since the structure within the sub blocks of $\Delta$ is lost. Moreover, for problems with $K = N = T$ greater than $60$, processor time for Backslash also includes memory management overhead, thus no longer a good proxy for arithmetic complexity. For problems with $K = N = T = 90$ and above, Backslash went out of memory. While the proposed PCGM performs much better with an added scope of the decomposibility of computations.

    Fig.~\ref{fig:SMD_conditioning} illustrates the actual condition number $\kappa(\Delta)$ and the achieved conditioning of the transformed system $\kappa(P_{S}^{-1/2} \Delta P_{S}^{-1/2})$ using the structured preconditioner $P_{S}$ given in Thm.~\ref{thm:posdef_preconditioner}. It can be seen that the conditioning is improved by almost two order of magnitude with the proposed approach.

    \textbf{Case (II): Automated Irrigation Network.} Fig.~\ref{fig:WIN_PCG_J_Iters} compares the total number of PCGM steps and standard NBJM steps with the size of the system. It can be seen that the number of steps performed by NBJM increases as $O(K^2)$ with the size of the problem reaching up to $10^5$ for $K=N=T=50$. As such, for larger systems, the results of NBJM are not included. By contrast, the increase in the number of steps of PCGM is $O(K)$ but the steps performed are much smaller than the size of the system.
    
    Fig.\ref{fig:WIN_log_Time_Normalized} shows the overall processor time of the three approaches. The processor time of NBJM increases as $O(K^5)$ and that of Backslash is $O(K^6)$. Whereas the increase in the processor time of the proposed PCGM is $O(K^4)$ which reflects the linear increase in the number of steps of PCGM along the line $K=N=T$. Moreover, for problem size larger than $K=N=T=40$, Backslash suffers from memory management issues and it went out of memory for problems larger than $K=N=T=60$. 

    Finally, Fig.\ref{fig:WIN_conditioning} shows the condition number of $\Delta$ in \eqref{eq:block_tridiagonal_system} and $\kappa(P_{S}^{-1/2} \Delta P_{S}^{-1/2})$ of the preconditioned system. It can be seen that there is an order of magnitude improvement in the conditioning with $L=2$ and $S=2$.

     \begin{figure}[!ht]
        \begin{center}
            \includegraphics[width = 0.48\textwidth]{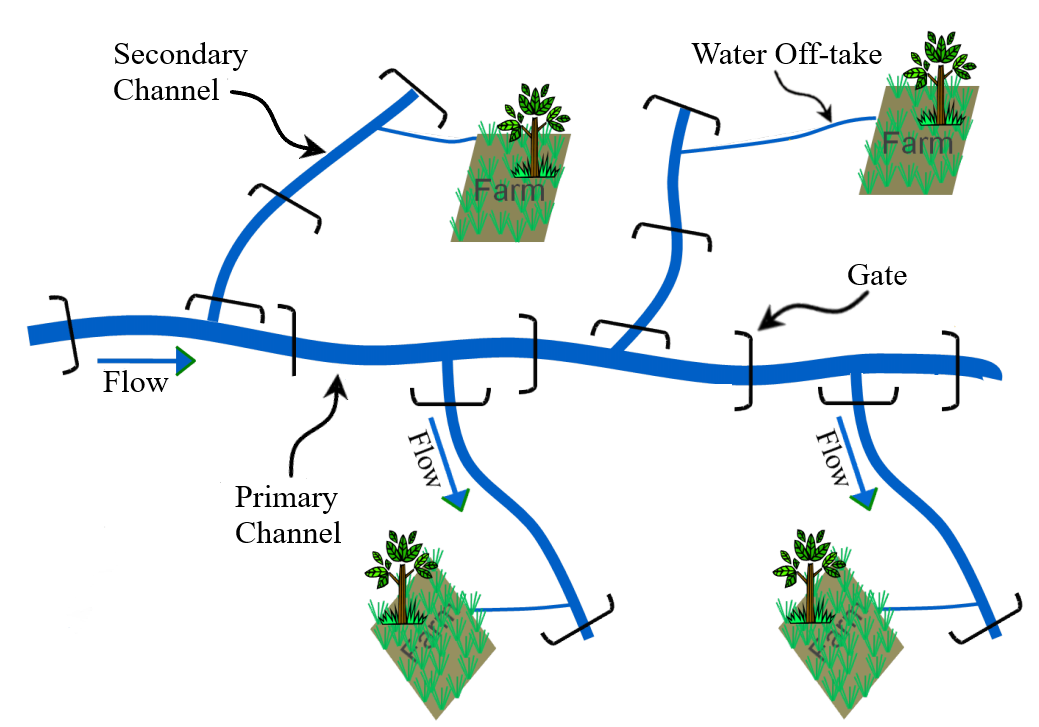} 
            \caption{A illustration of automated water irrigation network}
            \label{fig:water_irigation_network}
        \end{center}
    \end{figure}
    \begin{figure*}[htp]
        \centering
        \begin{minipage}[b]{0.32\textwidth}
            \centering
            \includegraphics[width=0.99\columnwidth,height=0.73\columnwidth]{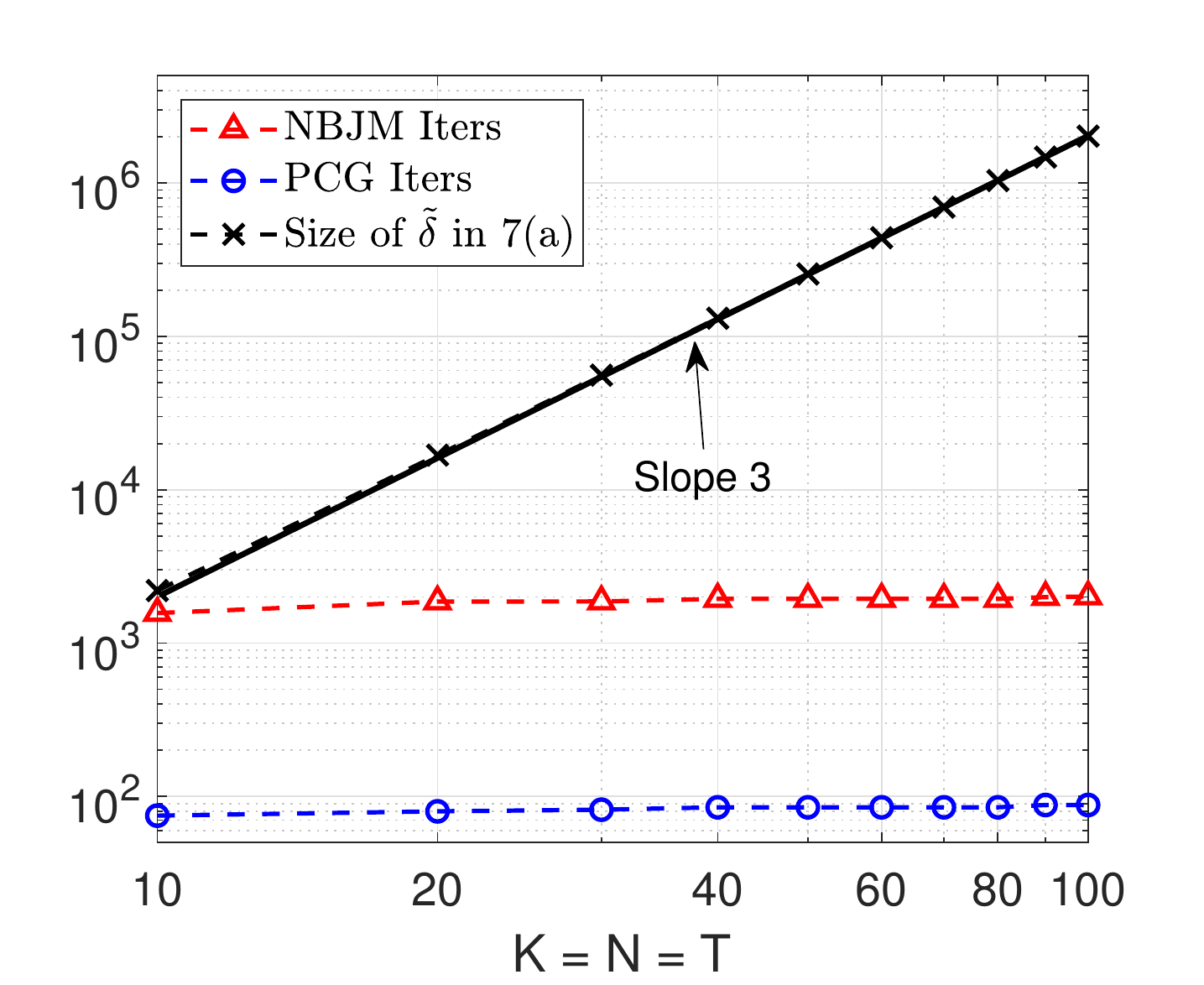}
            \caption{Case (I): Total no. of steps of PCGM and standard NBJM}
            \label{fig:SMD_PCG_J_Iters}
        \end{minipage}\hfill
        \begin{minipage}[b]{0.32\textwidth}
            \centering
            \includegraphics[width=0.99\columnwidth,height=0.73\columnwidth]{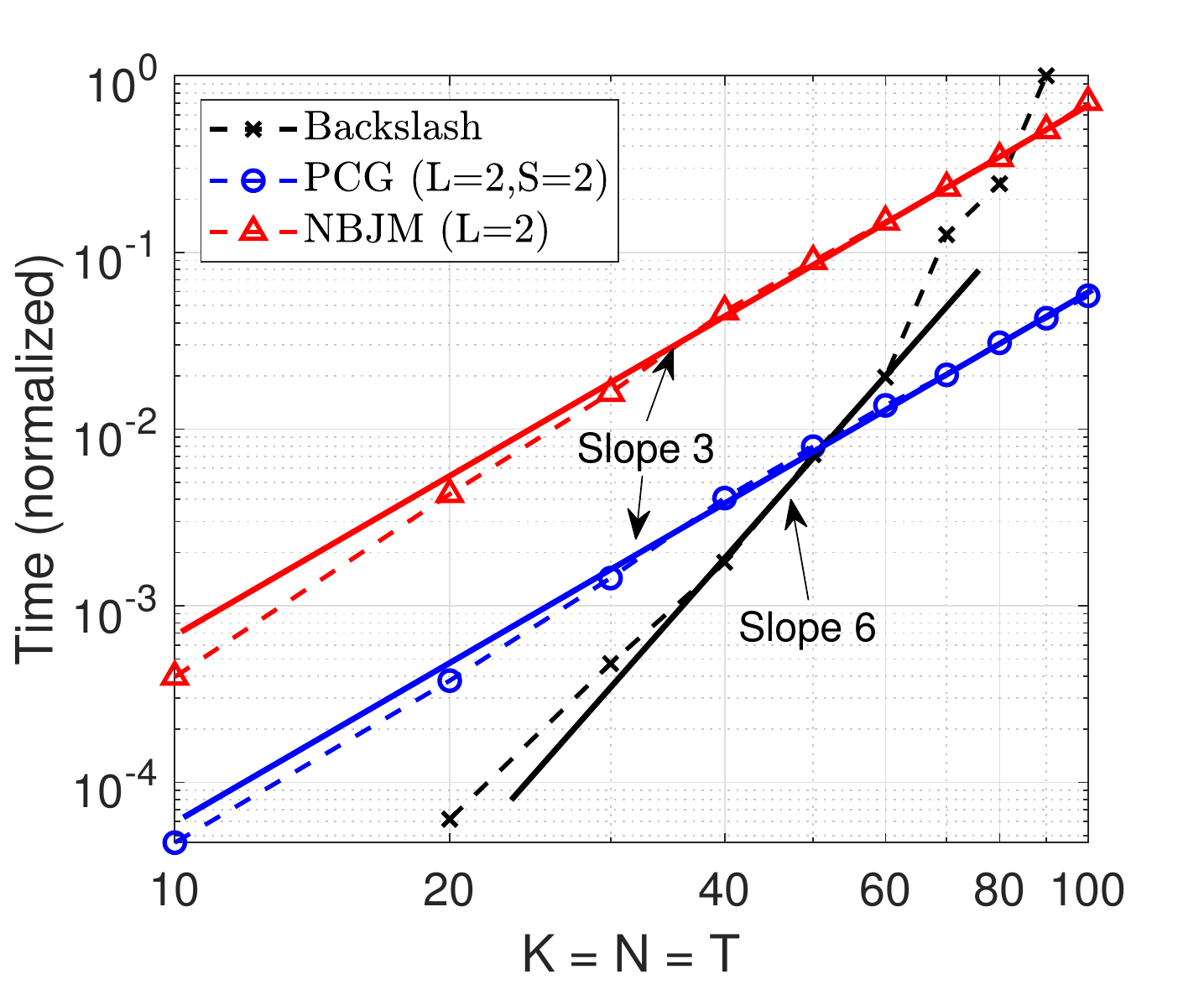}
            \caption{Case (I): Normalized processor time to solve \eqref{eq:block_tridiagonal_system}}
            \label{fig:SMD_log_Time_Normalized}
        \end{minipage}\hfill
        \begin{minipage}[b]{0.32\textwidth}
            \centering
            \includegraphics[width=0.99\columnwidth,height=0.73\columnwidth]{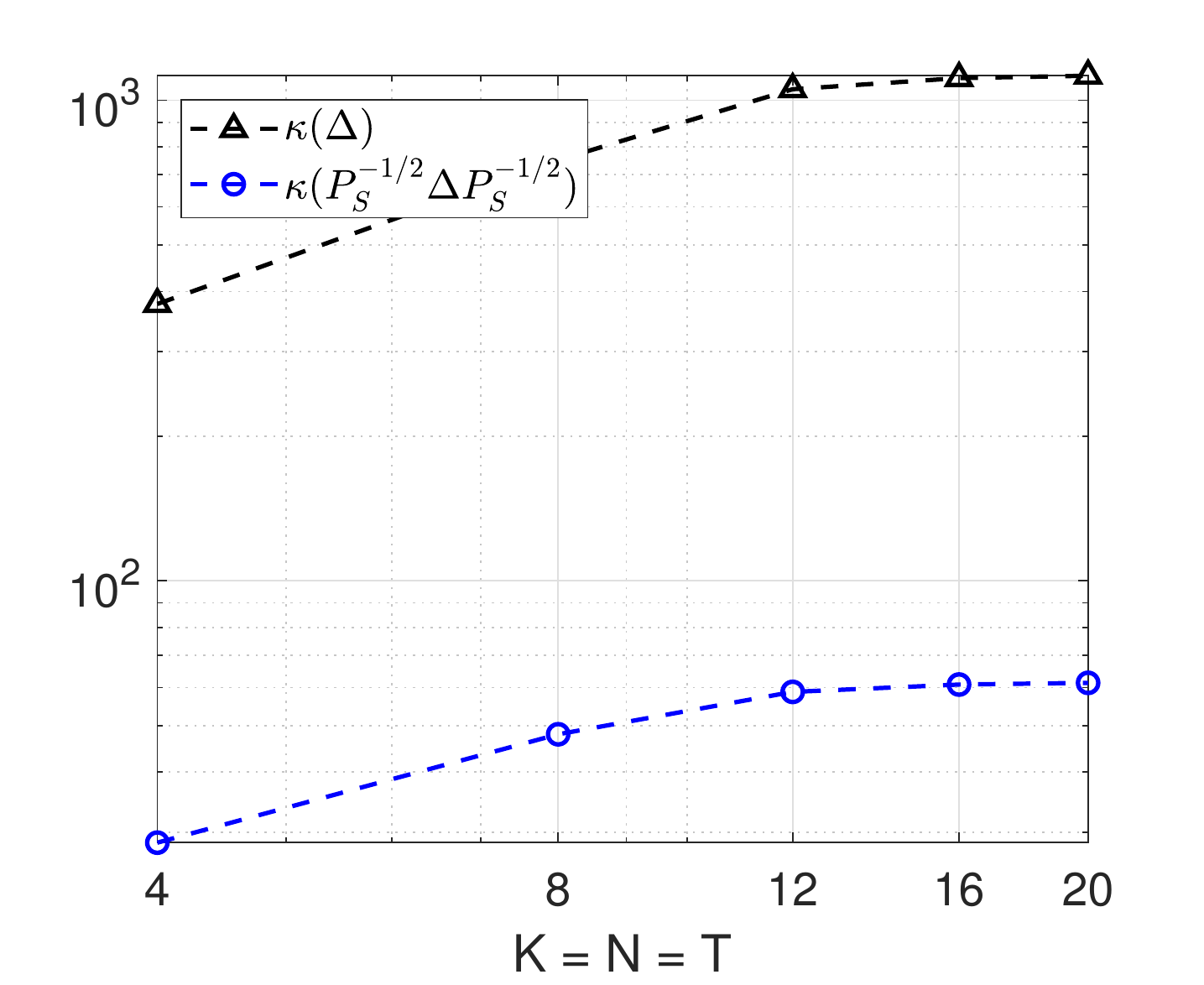}
            \caption{Case (I): $\kappa(\Delta)$ and $\kappa(P_{S}^{-1/2} \Delta P_{S}^{-1/2})$ with $L=2,~S=2$} 
            \label{fig:SMD_conditioning}
        \end{minipage}
    \end{figure*}
    \begin{figure*}[htp]
        \centering
        \begin{minipage}[b]{0.32\textwidth}
            \centering
            \includegraphics[width=0.99\columnwidth,height=0.73\columnwidth]{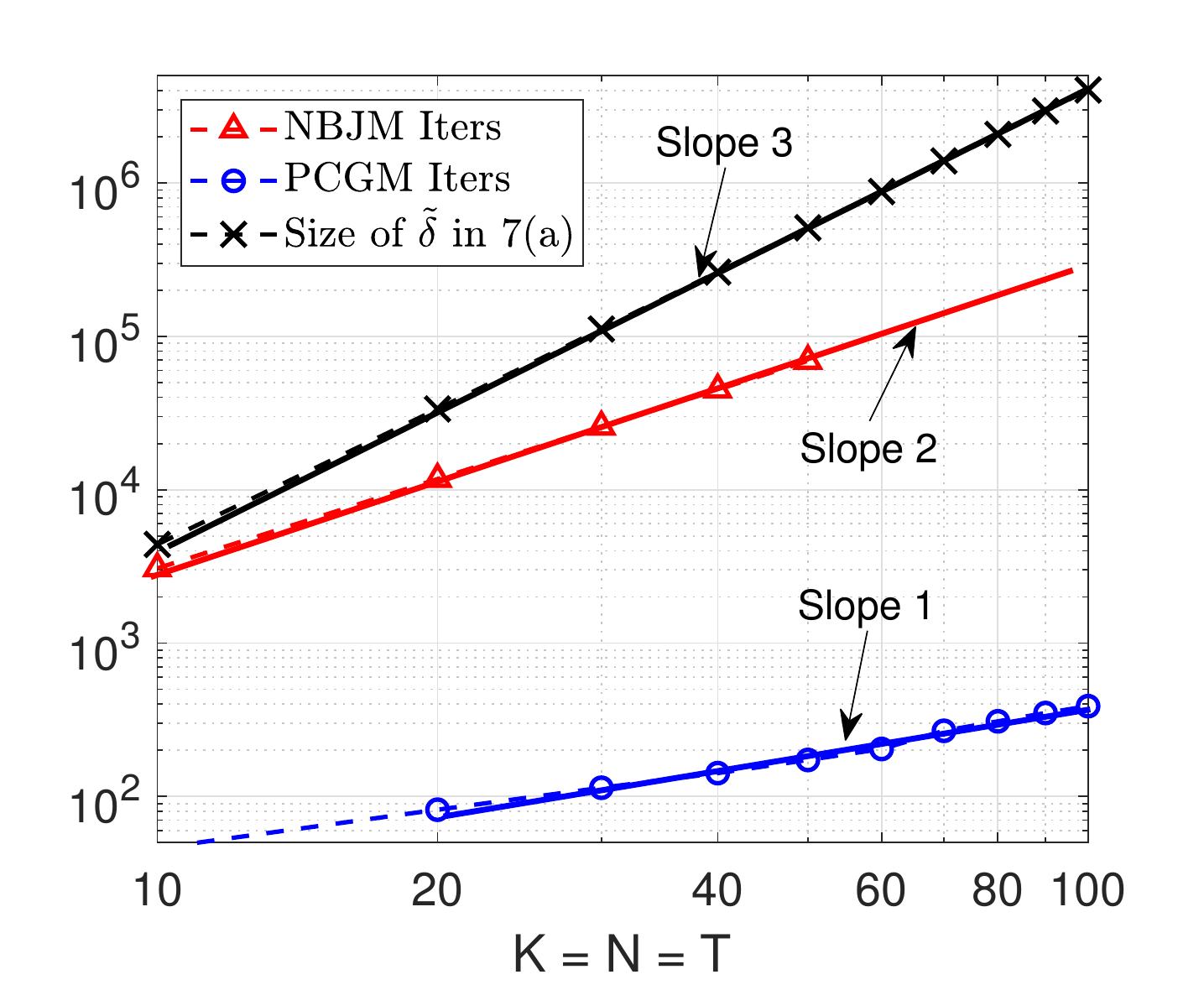}
            \caption{Case (II): Total no. of steps of PCGM and standard NBJM}
            \label{fig:WIN_PCG_J_Iters}
        \end{minipage}\hfill
        \begin{minipage}[b]{0.32\textwidth}
            \centering
            \includegraphics[width=0.99\columnwidth,height=0.73\columnwidth]{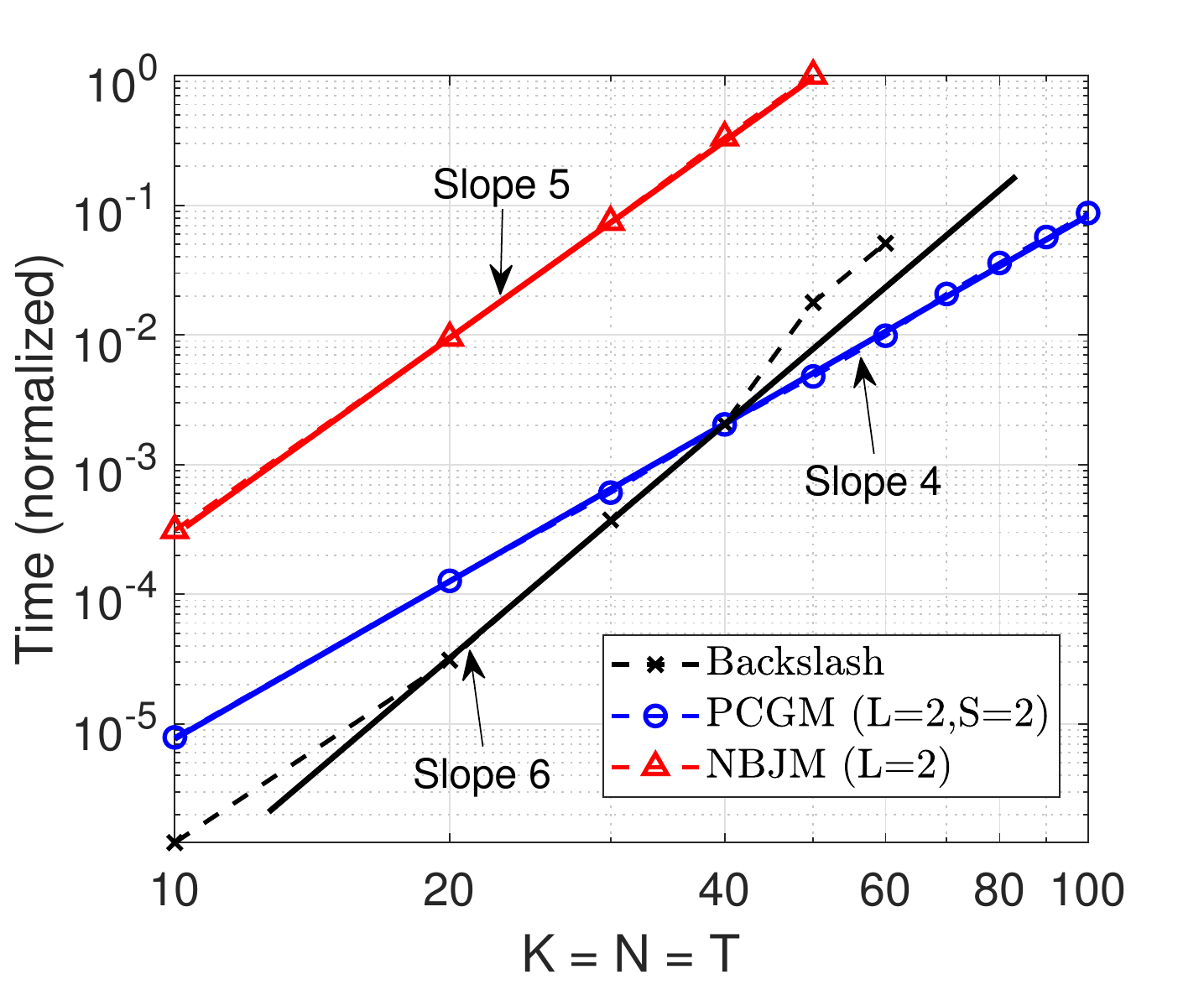}
            \caption{Case (II): Normalized processor time to solve \eqref{eq:block_tridiagonal_system}}
            \label{fig:WIN_log_Time_Normalized}
        \end{minipage}\hfill
        \begin{minipage}[b]{0.32\textwidth}
            \centering
            \includegraphics[width=0.99\columnwidth,height=0.73\columnwidth]{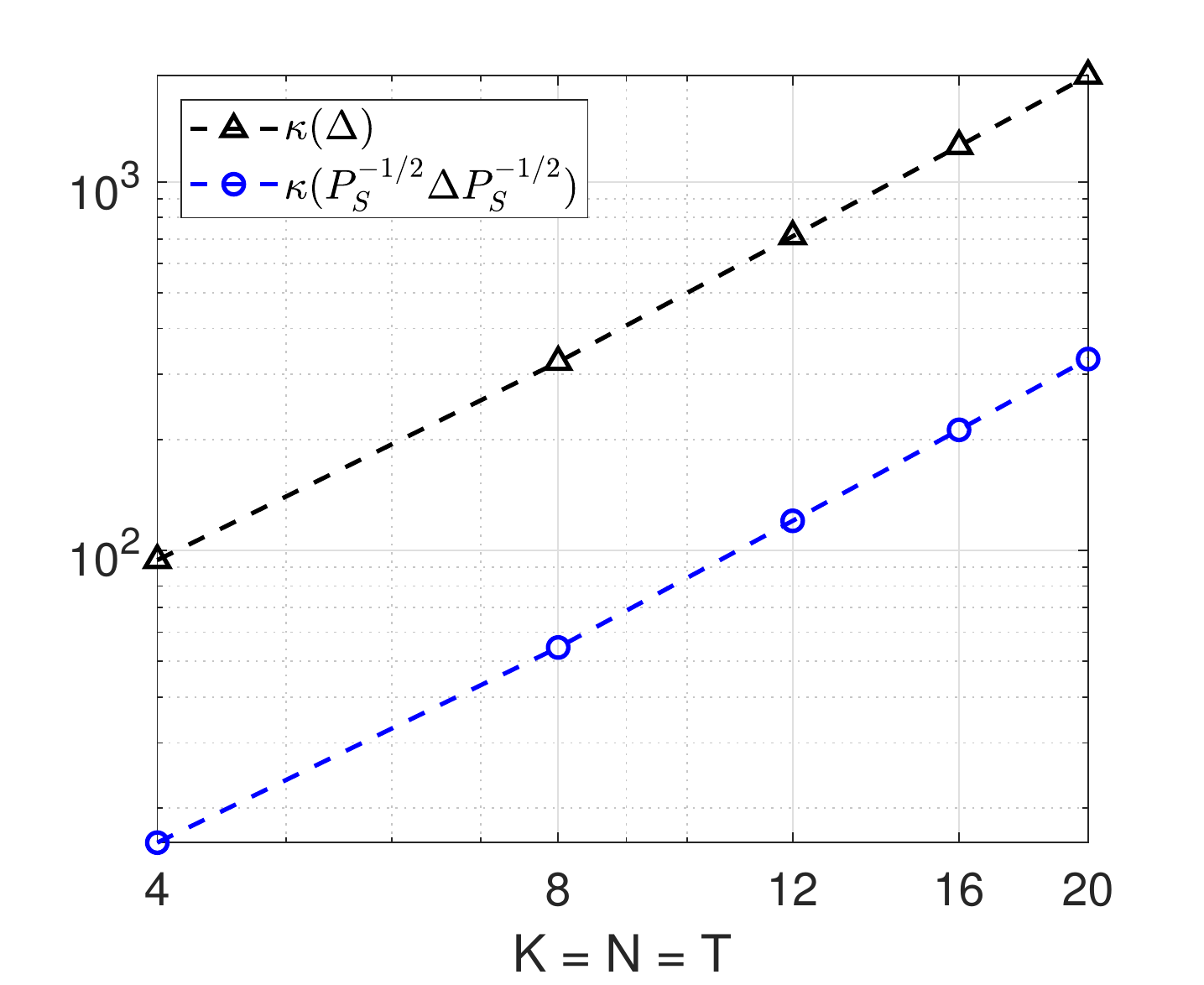}
            \caption{Case (II): $\kappa(\Delta)$ and $\kappa(P_{S}^{-1/2} \Delta P_{S}^{-1/2})$ with $L=2,~S=2$} 
            \label{fig:WIN_conditioning}
        \end{minipage}
    \end{figure*}

\section{Conclusions}\label{sec:conclusions}
    In this paper, we have proposed a structured PCGM for solving LQ optimal control problems of systems with 2D grid structure. The per step arithmetic complexity of the proposed approach scales linearly in each spatial dimension as well as in temporal dimension. The computations are amenable to distributed implementation on $O(NT)$ parallel processors with localized data exchange mirroring the 2D grid structure of the problem. Future works include the development of analytical bounds on the achieved conditioning and an implementation of the proposed approach on parallel processing architectures. It will enable us to perform a deeper analysis of the distributed performance of proposed approach along with the issues related with the data exchange overhead.

\section{Acknowledgement}\label{sec:acknowledgement}
    We would like to thank Prof. Michael Cantoni and Dr. Farhad Farokhi from the University of Melbourne, Australia for their insights and the helpful discussions about this work. This work was supported by the Air Force Office of Scientific Research Grant FA2386-19-1-4076 and the NSW Defence Innovation Network.  

\bibliographystyle{IEEEtran}

\end{document}